\documentclass[11pt,a4paper,twoside,notitlepage]{article}
\usepackage{amsmath, amsfonts, amsthm, siunitx, amssymb, mathtools, epstopdf, pdfpages}
\usepackage[utf8]{inputenc}
\usepackage[english]{babel}
\usepackage{graphicx}
\usepackage{caption}
\usepackage{subfigure}
\usepackage{float}
\usepackage{titling}
\usepackage[all]{xy}
\usepackage{subfiles, appendix}
\usepackage{listings}
\usepackage{graphicx}
\usepackage[left=3cm,right=3cm,top=4cm,bottom=4cm]{geometry}
\usepackage{listings, textcomp}
\usepackage{url}
\usepackage{graphicx,amsmath,amssymb,amsthm,siunitx,epstopdf,pdfpages,tikz, pgfplots,verbatim,listings,MnSymbol,subfigure,cite,caption,enumerate}
\usetikzlibrary{decorations.pathmorphing}
\usepackage{tikz-cd}
\usetikzlibrary{matrix, calc, arrows}
\usepackage{tabularx}
\usepackage{emptypage}
\usepackage{tikz}
\usetikzlibrary{shapes.misc}

\tikzset{cross/.style={cross out, draw=black, fill=none, minimum size=2*(#1-\pgflinewidth), inner sep=0pt, outer sep=0pt}, cross/.default={2pt}}
\newcolumntype{C}{>{\centering\arraybackslash}X}%
\usepackage{color}
\newtheorem{theorem}{Theorem}
\numberwithin{theorem}{section}
\newtheorem{definition}[theorem]{Definition}
\newtheorem{lemma}[theorem]{Lemma}
\newtheorem{corollary}[theorem]{Corollary}

\theoremstyle{definition}

\newtheorem{remark}{Remark}
\numberwithin{example}{section}
\numberwithin{remark}{section}

\usepackage[hidelinks]{hyperref}

\allowdisplaybreaks

\DeclareMathOperator{\G}{\mathcal{G}}
\DeclareMathOperator{\Hg}{\mathcal{H}}

\usepackage{afterpage}

\usepackage{xparse}
\NewDocumentCommand{\ceil}{s O{} m}{%
  \IfBooleanTF{#1} 
    {\left\lceil#3\right\rceil} 
    {#2\lceil#3#2\rceil} 
}

\title{Combinatorial cost: a coarse setting}
\author{Tom Kaiser}
\date{\today}

\begin{document}

\maketitle





\begin{abstract}
The main inspiration for this paper is a paper by Elek where he introduces combinatorial cost for graph sequences. We show that having cost equal to $1$ and hyperfiniteness are coarse invariants. We also show `cost$-1$' for box spaces behaves multiplicatively when taking subgroups. We show that graph sequences coming from  Farber sequences of a group have property A if and only if the group is amenable. The same is true for hyperfiniteness. This generalises a theorem by Elek. Furthermore we optimise this result when Farber sequences are replaced by sofic approximations. In doing so we introduce a new concept: property almost-A.
\end{abstract}

The introduction of cost finds its origins in measure equivalence theory, see \cite{gaboriau2016around} for a good reference. In this paper we will mostly study its combinatorial analogue, which was introduced by Elek in \cite{elekcombcost}. We first recall some of the necessary definitions and machinery.

A graph sequence $\mathcal{G}= \{G_n \}_{n\in\mathbb{N}}$ is a sequence of finite connected graphs $G_n$ of uniformly bounded degree $d$, such that $\lim\limits_{n\to\infty} \vert V(G_n)\vert = \infty$. Here $V(G_n)$ denotes the vertex set of $G_n$. Note that $E(G_n)$ denotes the edge set of $G_n$.
The edge number $e(\G)$ of a graph sequence is defined by 
\[ e(\G) = \liminf_{n\to\infty} \frac{\vert E(G_n)\vert}{\vert V(G_n)\vert}. \]
We also define the following equivalence relation of graph sequences: $\G\simeq \Hg$ if $V(G_n) =V(H_n)$ for all $n\in \mathbb{N}$ and the induced `identity map' on the vertices $\operatorname{Id}: \G\rightarrow \Hg $ is an $L$-bi-Lipschitz map for some $L>0$. This allows us to define the cost by the following infimum:
\[ c(\G) = \inf\limits_{\G\simeq\Hg } e(\Hg).\]
One can see that the edge number of a graph is at least one and thus its cost is too.

This paper looks at cost from a coarse viewpoint. An $(A,A)$-quasi-isometry between metric spaces $X$ and $X'$ is a map $f:X\rightarrow X'$ such that for every $x,y\in X$
\[ \frac{d(x,y)}{A} -A \le d(f(x),f(y)) \le A d(x,y) +A \]
and for every $x'\in X'$ there is an $x\in X$ such that $d(x',f(x))\le A$ (i.e. any point in $X'$ is $A$-close to an image point). More generally two such metric spaces are coarsely equivalent if for any two sequences $(x_n)_{n\in\mathbb{N}}$ and $(y_n)_{n\in\mathbb{N}}$ in $X$ we have that 
\[ d(x_n,y_n) \xrightarrow{n\rightarrow\infty} \infty \iff d(f(x_n),f(y_n)) \xrightarrow{n\rightarrow\infty} \infty \]
and there exists some constant $C$ such that the image of $f$ is $C$-dense in $X'$. The following lemma by Khukhro and Valette \cite{khukhro2015expanders} will justify our definition of coarsely equivalent graph sequences:

\begin{lemma}[Khukhro, Valette]\label{AA1}
Let $X = \displaystyle\bigsqcup_{k>0}^{+\infty} X_k$ and $Y = \displaystyle\bigsqcup_{k>0}^{+\infty} Y_k$ be coarse disjoint unions of graphs such that the diameter tends to infinity as $k$ tends to infinity and let $\Phi\colon X\to Y$ be a coarse equivalence between these metric spaces. Then there exists a constant $A$ and an almost permutation $\phi$ between the components of $X$ and the components of $Y$ such that $\Phi\vert_{X_i}$ is an $(A,A)$-quasi-isometry between $X_i$ and $\phi(X_i)$.
\end{lemma}

Two graph sequences $\G$ and $\Hg$ are coarsely equivalent if there exists a constant $A$ such that for all $n\in\mathbb{N}$ there exists an $(A,A)$-quasi-isometry $\phi_n: G_n\rightarrow H_n$. Note that we ignore the almost permutation part of the lemma since including this does not affect the value of the cost.\newline

In the remainder we show the following results, which are known for `ordinary' cost in the context of measure equivalence theory. We show this starting from the combinatorial setting. Note that the definition of hyperfiniteness will follow later in Section \ref{hyperfiniteness}.

\begin{theorem}
	If two graph sequences $\mathcal{G}$ and $\mathcal{H}$ are coarsely equivalent and one of them has cost $1$, then the other one also has cost $1$.
\end{theorem}
\begin{theorem}
	If $\G$ and $\Hg$ are coarsely equivalent and $\G$ is hyperfinite, then $\Hg$ is too.
\end{theorem}

Consider a finitely generated group $\Gamma$ with generating set $S$ and a subgroup $\Lambda\le \Gamma$. One can look at the set of right cosets $\Lambda\backslash \Gamma = \{\Lambda g\;\lvert\; g\in \Gamma \}$. The Schreier graph $Sch(\Gamma,\Lambda,S)$ on these cosets is defined as follows. The vertices are exactly the right cosets and there is a directed edge from $\Lambda g_1$ to $\Lambda g_2$ if there is a generator $s\in S$ such that $\Lambda g_1s= \Lambda g_2$. One can also label the edges by their defining generator. In this paper all Schreier graphs are labelled. Note that this graph is clearly connected. Also note that $Sch(\Gamma, \{1\}, S)$ is the well-known Cayley graph of $\Gamma$ with respect to $S$, written as $Cay(\Gamma,S)$. Constructing the Cayley graph of $\Gamma$ with respect to another generating set $S'$, gives a graph that is quasi-isometric to our original Cayley graph. Since we are working in a coarse setting, we will sometimes omit the generators and write $Cay(\Gamma)$. Moreover one sees that $Sch(\Gamma,\Lambda,S) = Cay(\Gamma/\Lambda,\overline{S})$ if $\Lambda$ is normal in $\Gamma$.\newline

Using this one can construct interesting graph sequences originating from $\Gamma$. For example take a descending sequence $N_1\ge N_2\ge \dots$ of finite index normal subgroups of $\Gamma$ such that $\bigcap\limits_{k\ge 1} N_k = \{1\}$. This is called a filtration of $\Gamma$ and $\Gamma$ is called residually finite if such a filtration exists. The box space with respect to this filtration is defined as $\square_{N_k} \Gamma = \{Cay(\Gamma/N_k, \overline{S} )  \}_{k\in \mathbb{N}}$. The box space does not depend on the choice of generating set in the sense that choosing a different generating set with respect to the same filtration will give a quasi-isometric box space. It is interesting to ask what such a box space can tell us about our group and vice versa. Elek for example showed the following theorem in \cite{elekcombcost}. The properties A and hyperfiniteness will be introduced in Sections  \ref{sectionA} and \ref{hyperfiniteness} respectively. The equivalence between $1$ and $2$ is originally due to Guentner and proven by Roe in \cite{roe2003lectures}.

\begin{theorem}[Elek,\cite{elekcombcost}]\label{elekstheorem}
Consider a residually finite group $\Gamma$ and a filtration $(N_k)_{k\in \mathbb{N}}$. If $\mathcal{G}$ is the associated graph sequence of Cayley graphs then the following are equivalent:
\begin{enumerate}
	\item $\Gamma$ is amenable;
	\item $\mathcal{G}$ has property A;
	\item $\mathcal{G}$ is hyperfinite.
\end{enumerate}
\end{theorem}

We will introduce so-called Benjamini-Schramm convergence in the restrictive setting where the limit object is a Cayley graph. One will see that a box space Benjamini-Schramm converges to the Cayley graph of the group from which it originates.  For a nice reference on Benjamini-Schramm convergence one can consult \cite{abert2013benjamini}.\newline

Fix an integer $d$. Let $\mathcal{RG}_d$ be the set of all isomorphism classes of finite rooted graphs of degree $d$ (i.e. a graph with a distinguished vertex that acts as the root), with edges labelled by $\{1,\dots, d\}$. 
Fix a finite graph $G$ and $r>0$. We can define the following probability measure $p_r: \mathcal{RG}_d\rightarrow [0,1]$ for the graph $G$.  For $\alpha\in\mathcal{RG}_d$, we define $V_\alpha(G) = \{ x\in V(G)\; \lvert B_r(G,x) \simeq \alpha \}$ to be vertices of $G$ such that the $r$-ball around such a vertex is isomorphic to $\alpha$. Next we define $p_r(G,\alpha)= \frac{\lvert V_\alpha(G)\rvert}{\lvert V(G)\rvert}$. Hence we look at the probability that a uniformly chosen vertex of $G$ has its $r$-ball isomorphic to $\alpha$. \newline

We say that a labelled graph sequence $\G = \{G_n\}_{n\in\mathbb{N}}$ with labels in $S$ converges in a Benjamini-Schramm sense to the Cayley graph of a group $\Gamma$ with generating set $S$ if for every $r\in\mathbb{N}$, the limit $\lim\limits_{n\rightarrow \infty} p_r(G_n,B_r(Cay(\Gamma),e)) =1$, where $B_r(Cay(\Gamma),e)$ is the $r$-ball in $Cay(\Gamma)$ around the identity element. So in particular this limit exists.\newline

An immediate consequence of this definition is that box spaces converge in a Benjamini-Schramm sense to the Cayley graph of the defining group. However, inspired by this definition, we can do a bit better. Let's look at an arbitrary sequence of finite index subgroups (not necessarily normal, not necessarily with trivial intersection, not necessarily a descending chain). This induces a sequence of Schreier graphs associated to the group $\Gamma$. We say the defining sequence of subgroups is a Farber sequence if the associated graph sequence is Benjamini-Schramm convergent to the Cayley graph of $\Gamma$. See \cite{rankgradcost} for a nice reference. For these `generalisations of box spaces' we can also ask the natural question whether certain properties of the group induce certain properties of the graph sequence and vice versa. In particular we show Elek's Theorem \ref{elekstheorem} in the more general setting of Farber sequences.
\begin{theorem}\label{Farbertheo}
	Consider a finitely generated group $\Gamma$ and a Farber sequence $(N_k)_{k\in \mathbb{N}}$. If $\mathcal{G}$ is the associated graph sequence of Schreier graphs then the following are equivalent:
	\begin{enumerate}
		\item $\Gamma$ is amenable;
		\item $\mathcal{G}$ has property A;
		\item $\mathcal{G}$ is hyperfinite.
	\end{enumerate}
\end{theorem}

One can notice that we can still go a step further and even forget about subgroups alltogether. We look at arbitrary labelled graph sequences that converge to $\Gamma$ in a Benjamini-Schramm sense. If such a graph sequence exists we say the group is sofic. It is not known whether there exists a non-sofic group. This concept was first introduced by Gromov \cite{gromov1999endomorphisms}, but was named by Weiss \cite{weiss2000sofic}. We shall argue in Section \ref{elekfarbersection} that Theorem \ref{Farbertheo} is optimal and cannot be translated to the sofic approximation case. In \cite{alekseev2016sofic} Alekseev and Finn-Sell use a groupoid approach to show the more general implication that the group is amenable if it has a sofic approximation with property A. The reverse implication does not hold. We show the same holds for hyperfiniteness. The reverse implication does also hold. Although not explicitly presented in \cite{elekcombcost}, this result is due to Elek.
\begin{theorem}
Consider a group $\Gamma$ and a sofic approximation $\mathcal{G}$ of $\Gamma$. The following holds
	\begin{enumerate}
		\item If $\G$ has property A, then $\Gamma$ is amenable (Alekseev, Finn-Sell, \cite{alekseev2016sofic}).\\
				The reverse arrow does not hold.
		\item If $\G$ is hyperfinite, then $\Gamma$ is amenable.\\
				The reverse arrow holds (Elek, \cite{elekcombcost}).
	\end{enumerate}
\end{theorem}

Sadly enough this generalisation to the sofic approximation case breaks the symmetry. This motivates us to introduce property almost-A in Definition \ref{almosta}. This `saves' the theorem.
\begin{theorem}
	Consider a group $\Gamma$ and a sofic approximation $\G$ of that group, then the following are equivalent:
	\begin{enumerate}
		\item $\Gamma$ is amenable;
		\item $\mathcal{G}$ has property almost-A;
		\item $\G$ is hyperfinite.
	\end{enumerate}
\end{theorem}
We show the following result for general graph sequences. It is a generalisation of a result in \cite{alekseev2016sofic}, where it is proved for sofic approximations.

\begin{theorem}
	A graph sequence $\G$ that has property A is also hyperfinite.
\end{theorem}


In Section \ref{costmulttheo} we show that cost of graph sequences coming from Farber sequences behave multiplicatively with respect to taking finite index subgroups.

\begin{theorem}\label{supercostmult}
	Take $\Gamma$ a residually finite group and $\Lambda\le \Gamma$ with $[\Gamma:\Lambda]< \infty$. Consider a Farber sequence $(N_k)_{k\in\mathbb{N}} $ of $\Gamma$ that is contained in $\Lambda$. Define $\G = \{Sch(\Gamma,N_k,S)\}_{k\in\mathbb{N}}$ and $\Hg= \{Sch(\Lambda,N_k,S')\}_{k\in\mathbb{N}} $ then
	$$ [\Gamma:\Lambda] (c(\G)-1) =  c(\Hg)-1.$$ 
\end{theorem}
 


\section{Cost 1 is a coarse invariant}

Let us start with a little remark. In order be an interesting invariant, it is important to see that the cost can attain more than only one value. Hence we give an example of a graph sequence where the cost is not one. Consider a graph sequence $\mathcal{G} = \{G_n \}_{n\in\mathbb{N}}$. The girth $gr(G_n)$ of a graph $G_n$ is the length of the smallest non-trivial cycle in this graph. A graph sequence has large girth if $ \lim\limits_{n\rightarrow \infty} gr(G_n)$ exists and tends to infinity. Elek shows in \cite{elekcombcost} that if a graph sequence has large girth, then its edge number is equal to its cost. The free group on $r$ generators $F_r$ is residually finite and thus it has a filtration, say $\{N_k\}_{k\in\mathbb{N}}$. It is easy to see that the box space $\square_{N_k} F_r$ has large girth and that for every vertex there are $r$ edges. Hence $c(\square_{N_k} F_r) = e(\square_{N_k} F_r) = r$.\newline

Now that we have indicated that having cost $1$ is not trivial, we show it is a coarse invariant. Consider a graph sequence $\mathcal{G} = \{G_n \}_{n\in\mathbb{N}}$. Fix $k\ge 0$. We construct $\mathcal{G}'$, which is the graph sequence $\mathcal{G}$ where we add at most $k$ vertices to each vertex. More rigorously: for every vertex $x$ in some $G_n$ we can add at most $k$ vertices $x_1, \dots, x_k$ and $k$ edges $(x,x_1),\dots,(x,x_k)$ to obtain $G_n'$. We will call these added vertices terminal (or terminal vertices). One can see that $c(\mathcal{G})=1$ if and only if $c(\mathcal{G}') = 1$. We show this by proving that minimising the edge number of $\G'$ through equivalent graph sequences is the same as minimising the edge number of $\G$ and leaving the terminal vertices untouched.

\begin{lemma}
Given a graph sequence $\G'$ and the graph sequence $\G$ obtained by deleting the terminal vertices (i.e. vertices of degree $1$) of $\G'$, there exists $\forall \epsilon>0$ a graph sequence $\G'_\epsilon$ such that
\begin{enumerate}
	\item $\G'_\epsilon \simeq \G'$;
	\item $e(\G'_\epsilon)\le c(\G') + \epsilon$;
	\item if $(x,y)$ is an edge in $\G'_\epsilon$, but not in $\G'$ (or vice versa), then $x,y\in \G$ (i.e. $x$ and $y$ are not terminal vertices).
\end{enumerate}
\label{lemma1}
\end{lemma}

\begin{proof}
	Note that by the definition of cost, graph sequences $\Hg'_\epsilon$ (these are $L'_\epsilon$-bi-Lipschitz to $\G'$) that satisfy point $1$ and $2$ trivially exist, it is enough to show that from such sequences we can construct new ones that have property $3$ but keep properties $1$ and $2$. We will use the following terminology. A vertex in $\G$ will be called an internal vertex. The  vertex in $\G$ that connects a terminal vertex $y$ (in $\G'$) to $\G$ will be called the base vertex of $y$. Now notice that we have the following types of edges in $\Hg'_\epsilon$:
	\begin{enumerate}
		\itemsep0em
		\item edges between terminal vertices;
		\item edges between internal vertices;
		\item edges between terminal vertices and internal vertices.
	\end{enumerate}
	We show that $\Hg_\epsilon'$ can be chosen such that it does not contain edges of type one. Suppose an edge $(y_1,y_2)$ is an edge between terminal vertices in $\Hg'_\epsilon$ such that $y_1$ is connected to an internal vertex $x$, then we remove $(y_1,y_2)$ and add the edge $(x,y_2)$ (see Figure \ref{type1}). One can do this for all such edges and obtain a graph sequence that is $2$-bi-Lipschitz equivalent to $\Hg'_\epsilon$. Now, since between any terminal vertex and its base vertex there is a path of length at most $L'_\epsilon$ in  $\Hg'_\epsilon$, any edge between terminal vertices is $L'_\epsilon$-close to an edge between a terminal vertex and an internal vertex. This means that repeating this process at most $L'_\epsilon$ times gives us a graph sequence that is $2L_\epsilon ' $-bi-Lipschitz to $\Hg_\epsilon'$. Hence we can assume that $\Hg_\epsilon'$ does not have edges of type one.
	
	\begin{minipage}{\linewidth}
		\centering
		\begin{minipage}{0.47\linewidth}
			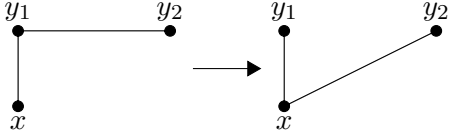
\begin{figure}[H]
				\begin{tikzpicture}[>=triangle 60]
				\draw[fill=black] (0,-1) circle (2pt) node[align=left,  below] {$x$} --
				(0,0) circle (2pt) node[align=center,   above] {$y_1$} --
				(2,0) circle (2pt) node[align=center, above] {$y_2$};
				\draw[->] (2.3,-0.5) -- (3.2,-0.5);
				\draw[fill=black] (3.5,0) circle (2pt) node[align=center,   above] {$y_1$} -- 
				(3.5,-1) circle (2pt) node[align=left,  below] {$x$} --
				(5.5,0) circle (2pt) node[align=center, above] {$y_2$};
				\end{tikzpicture}
				\caption{Construction for edges\\\hspace{\textwidth} of type one.}
				\label{type1}
			\end{figure}
		\end{minipage}
		\begin{minipage}{0.47\linewidth}
			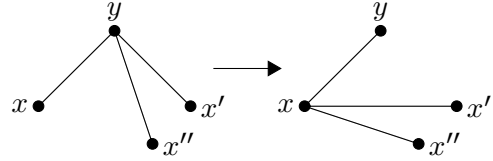
\begin{figure}[H]
				\begin{tikzpicture}[>=triangle 60]
				\draw[fill=black] (0,-1) circle (2pt) node[left] {$x$} --
				(1,0) circle (2pt) node[align=center,   above] {$y$} --
				(2,-1) circle (2pt) node[right] {$x'$};
				\draw[fill=black] (1,0) circle (0.1pt) --
				(1.5,-1.5) circle (2pt) node[right] {$x''$};
				\draw[->] (2.3,-0.5) -- (3.2,-0.5);
				\draw[fill=black] (4.5,0) circle (2pt) node[align=center,   above] {$y$} --
				(3.5,-1) circle (2pt) node[left] {$x$} --
				(5.5,-1) circle (2pt) node[right] {$x'$};
				\draw[fill=black] (3.5,-1) circle (0.1pt) --
				(5,-1.5) circle (2pt) node[right] {$x''$};
				\end{tikzpicture}
				\caption{Construction for edges\\\hspace{\textwidth} of type three (a).}
				\label{type2}
			\end{figure}
		\end{minipage}
	\end{minipage}
	\newline

	Now we look at edges of type three. Take a terminal vertex $y$. If it is connected to more than one internal vertex, then we choose one of these vertices (for example $x$) and reconnect as follows: take any other vertex connected to $y$ (for example $x'$). Then we delete the edge $(x',y)$ and introduce the edge $(x',x)$. (see Figure \ref{type2}). Doing this for all terminal vertices gives a graph sequence that is $2$-bi-Lipschitz to $\G$. We can assume terminal vertices in $\Hg_\epsilon'$ are connected to at most one internal vertex.
	
	Now it is possible that some terminal vertices are not linked to their base vertex $x$, but to some other internal vertex $x'$. However since there is no edge between the terminal vertex $y$ and its base vertex $x$, there is a path of length at most $L_\epsilon '$ between them. The first edge of this path links $y$ to some internal vertex $x'$. We delete the edge $(x',y)$ and add the edge $(x,y)$. See Figure \ref{type3}. It is clear this gives a graph sequence that is $L_\epsilon '$-bi-Lipschitz to $\Hg_\epsilon '$. Note that we can do this independently for all such edges (since apart from the first edge these paths are in $\G$). Note that this last graph sequence has all the needed properties. Replacing $\Hg_\epsilon'$ by this graph sequence concludes the proof.
	
	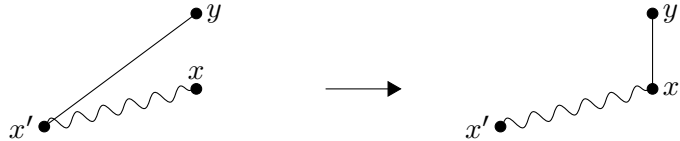
\begin{figure}[H]
		\centering
		\begin{tikzpicture}[>=triangle 60]
		\draw[fill=black] (2,0) circle (2pt) node[right] {$y$}
		(2,-1) circle (2pt) node[align=center,   above] {$x$}
		(0,-1.5) circle (2pt) node[left] {$x'$};
		\draw (2,0) -- (0,-1.5);
		\draw[style={decorate, decoration=snake}] (0,-1.5)--(2,-1);
		\draw[->] (3.7,-1) -- (4.7,-1);
		\draw[fill=black] (8,0) circle (2pt) node[ right] {$y$}
		(8,-1) circle (2pt) node[align=center,   right] {$x$}
		(6,-1.5) circle (2pt) node[left] {$x'$};
		\draw (8,-1) -- (8,0);
		\draw[style={decorate, decoration=snake}] (6,-1.5)--(8,-1);
		\end{tikzpicture}
		\caption[superthing]
		{Construction for edges of type three (b).}
		\label{type3}
	\end{figure}

\end{proof}
\begin{remark}
	Note that we do not have to delete all terminal vertices, we can also just delete subsets of terminal vertices with the same result.
\end{remark}

\begin{corollary}\label{cor2}
Given a graph sequence $\G'$ and the graph sequence $\G$ obtained by deleting the terminal vertices (i.e. vertices of degree $1$) of $\G'$, then $\G'$ having cost $1$ is equivalent to $\G$ having cost $1$.
\end{corollary}
\begin{proof}
If $c(\mathcal{G})=1$, we can find equivalent graph sequences $\mathcal{G}_\epsilon$ such that $e(\mathcal{G}_\epsilon)\le 1+\epsilon$. It is clear that changing the copy of $\G$ in $\G'$ by $\G_\epsilon$ yields a graph sequence $\G'_\epsilon$ equivalent to $\G'$ which has edge number lesser or equal than $1+\epsilon$.\newline

On the other hand, if $c(\G')=1$, then we can find equivalent graph sequences $\G'_\epsilon$ such that $e(\G'_\epsilon)\le 1+\epsilon$ and moreover, such that the changes in $\G'_\epsilon$ only appear in the $\G$-part. This is due to Lemma \ref{lemma1}. This means that reducing  $\G'_\epsilon$ to the vertices of $\G$ gives a graph sequence that is bi-Lipschitz to $\G$. We calculate the edge number of this graph sequence we will call $\G_{d\epsilon} = \{ G_{d\epsilon,n } \}_{n\in\mathbb{N}}$.  Note that the expression $C_\mathbb{Q} (G_{d\epsilon,n })$ denotes the vector space generated by the cycles of $G_{d\epsilon,n }$ over the field $\mathbb{Q}$. Moreover since $\G_{d\epsilon}$ is just $\G'_\epsilon$ without some terminal vertices, one sees that $\dim_{\mathbb{Q}} C_\mathbb{Q} (G_{\epsilon,n }') = \dim_{\mathbb{Q}} C_\mathbb{Q} (G_{d\epsilon,n })$. Note that since the degree of $\G'$ is given by $d$, we can bound the number of vertices of $G_n'$ from above by $d|V(G_n)|$. This allows us to calculate:
\begin{equation*}
 \begin{aligned}
 e(\G_{d\epsilon}) &= \liminf_{n\rightarrow \infty} \frac{|E(G_{d\epsilon,n})|}{|V(G_{d\epsilon,n})|}\\
 &= \liminf_{n\rightarrow \infty} \frac{|V(G_{d\epsilon,n})| + \dim_{\mathbb{Q}} C_\mathbb{Q} (G_{d\epsilon,n}) }{|V(G_{d\epsilon,n})|}\\
 &\le 1 + \liminf_{n\rightarrow \infty} \frac{\dim_{\mathbb{Q}} C_\mathbb{Q} (G'_{\epsilon,n})}{|V(G'_{\epsilon,n})|/d}\\
 &\le 1 + d\epsilon.
 \end{aligned}
\end{equation*}
This finishes the proof.
\end{proof}
We now move on to show that cost $1$ is an invariant of coarse equivalence. We do this by constructing coarse equivalences between the graph sequences of the form we described above.

\begin{theorem}
If two graph sequences $\mathcal{G}$ and $\mathcal{H}$ are coarsely equivalent and one of them has cost $1$, then the other one also has cost $1$. \label{cost1lemma}
\end{theorem}
\begin{proof}
We take a coarse equivalence $f:\mathcal{G}\rightarrow \mathcal{H}$ and suppose $c(\mathcal{G})=1$. Note that by throwing out graphs and reordering the sequences we can assume that restricting $f$ gives an $(A,A)$-quasi-isometry $f_n:G_n\rightarrow H_n$ for every $n$ with uniform $A$ (Lemma \ref{AA1}).
The map $f$ is not necessarily injective, but we know that all vertices in $G_n$ that have the same image over $f$ must be inside the ball $B(x,A^2)$, where $x$ is one of these vertices. Note that by the bounded degree of the graphs, we can find a bound on the number of elements inside this ball (say $C$). We construct a graph sequence $\mathcal{H}'$ and a map $f':\G\rightarrow \Hg'$ as follows: take $y\in H_n$. If $|f^{-1}(y)|>1$, then we add $|f^{-1}(y)|-1$ points to $H_n$ and connect these to $y$. Now we construct a new map $f'$ by bijectively mapping each of these $|f^{-1}(y)|-1$ points to one of the newly added points in $\Hg'$.

\begin{figure}[H]
\centering
\begin{tikzpicture}
\draw[fill=black] (2,0) circle (2pt) node[align=left,  above] {$y$}
(3.4,1.4) circle (2pt) node[align=center,   right] {$y_1$} 
(4,0) circle (2pt) node[align=center, right] {$y_2$}
(3.4,-1.4) circle (2pt) node[align=center,   right] {$y_3$}
(0.2,-0.8) circle (2pt) node[align=center,   left] {$y_{|f^{-1}(y)|-1}$};
\draw  (2,0) -- (3.4,1.4)
(2,0) -- (4,0)
(2,0) -- (3.4,-1.4)
(2,0) -- (0.2,-0.8);
\draw (1.5,-1.2) node{$\dots$};
\end{tikzpicture}
\caption[superthing]
{Construction of $\mathcal{H}'$.}
\label{thing}
\end{figure}
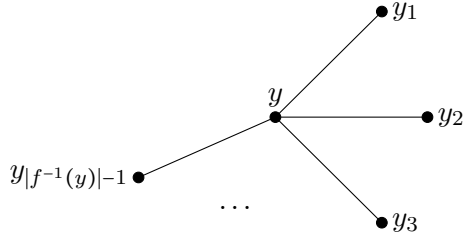

As such we obtain $\mathcal{H}'$ (with new degree $d+C$ where $d$ is the original degree of $\Hg$) and $f'$ (with some constant $A'$). Firstly remark that $\Hg$ has cost $1$ if and only if $\Hg'$ has, by Corollary \ref{cor2}. Note also that the new coarse equivalence $f': \mathcal{G}\rightarrow \mathcal{H}' $ is injective. From this we will can construct $\mathcal{H}''$ on the vertices of $\Hg'$. Take any two vertices in $\mathcal{G}$ connected by an edge. By injectivity these vertices have different images. We introduce an edge between these image points. For vertices of $\Hg'$ that are not in the image we break all edges that end in this vertex and create an edge that links it to the closest image point (in a similar way as when we constructed $\Hg'$, but now we reconnect existing vertices instead of creating new ones). One sees that $\Hg''$ is just $\mathcal{G}$ with extra vertices attached. We show that $\mathcal{H}'\simeq \mathcal{H}''$. 


Take an edge in $\mathcal{H}''$, that is not in $\mathcal{H}'$. If this edge links a vertex in the image of $f'$ to one that is not in the image of $f'$, then there must be a path of length less than or equal to $R$ (where $R$ is a constant bounding how far non-image points are from image points of $f'$) in $\mathcal{H'}$ between these two vertices since $f'$ is a coarse equivalence. If on the other hand this is an edge between two image points, then this means that the inverse images of these vertices in $\G$ are only separated by distance one. The function $f'$ being an $A'$-coarse equivalence means that the original two vertices are at most distance $2A'$ apart in $\Hg'$.

Now take an edge $(x,y)$ in $\mathcal{H}'$. There are three possibilities:
\begin{enumerate}
\itemsep0em
\item  The two vertices are not image points. This means that either they are connected to the same image point in $\mathcal{H}''$, and so they are only distance two apart. Or these are connected to two different image points, which are at most $2R+1$ apart in $\Hg'$. As such their pre-images are at most $A'(2R+1) +A'^2$ apart. By construction of $\Hg''$ (i.e. $\Hg''$ contains a copy of $\G$) the same path exists between the image points. Hence $x$ and $y$ are at most $A'(2R+1) +A'^2+2$ apart in $\mathcal{H}''$.
\item  One vertex is an image point and the other is not. See part one, but now the image points are at most distance $R+1$ apart in $\Hg'$ and thus our original vertices are at most  $A'(R+1) +A'^2+1$ apart in $\mathcal{H}''$.
\item  The two vertices are image points. There is a path of length at most $A'+A'^2$ between their pre-images and thus between $x$ and $y$.
\end{enumerate}
Taking $L$ the maximum of all these values, we find that $\Hg'\simeq \Hg''$ by the constant $L$. Once again using Corollary \ref{cor2} and the fact that $\G$ has cost $1$, we obtain that  $\Hg''$ also has cost $1$. This finishes the proof.
\end{proof}

\section{Hyperfiniteness}\label{hyperfiniteness}

In \cite{elekcombcost} Elek introduces hyperfiniteness as follows:
\begin{definition}
	A graph sequence $\G = \{G_n \}_{n\in \mathbb{N}}$ is hyperfinite if $\forall n \in \mathbb{N}$ and $\forall \epsilon >0$ there exist sets $E_n^\epsilon\subset E(G_n)$ and a constant $K>0$ such that 
	\begin{enumerate}
		\item $\lim\limits_{n\to\infty} \frac{\lvert E_n^\epsilon\rvert }{\lvert V(G_n)\rvert} < \epsilon$,
		\item $\forall n\in \mathbb{N}$ one has that the connected components of $G_n\backslash E_n^\epsilon$ contain at most $K$ vertices.
	\end{enumerate}
\end{definition}

We show that hyperfiniteness is a coarse invariant. Firstly we show it is preserved by bi-Lipschitz equivalence of graph sequences.

\begin{lemma}
If $\mathcal{G}\simeq \G'$ and $\G$ is hyperfinite, then $\G'$ is as well.
\end{lemma}
\begin{proof}
We look at the `identity map' $\operatorname{Id}:\G\rightarrow \G'$ which is an $L$-bi-Lipschitz map for some $L>0$. Take $\epsilon >0$. Since $\G$ is hyperfinite there exists a partition $\{A^n_i\}$ where $n$ ranges over the positive integers and $i$ is a value in between one and a constant $k_n\in \mathbb{N}$ which depends on $n$. There is a constant $K$ that bounds the number of elements in these sets. Moreover the partition $\{A^n_i\}$ induces sets $E_n^\epsilon$ that contain all edges between these partitions.  Note that we can interpret going from $\G$ to $\G'$ as first deleting some edges and then adding some. Clearly deleting edges does not form any problems and does not affect hyperfiniteness, since we can just copy the $A^n_i$ and $K$. However adding edges could result in destroying the structure of the graph and thus losing this property. How many edges (proportionally to the vertices in the individual graphs) can we actually add for an arbitrary graph sequence and still have that $\operatorname{Id}$ is $L$-bi-Lipschitz (we take $L$ fixed here)? Note that we can only add an edge $(x,y)$ between components (i.e. $x\in A_i^n$ and $y\in A_j^n$, where $i\ne j$) if there is a path in $G_n$ between the vertices $x$ and $y$ of length less than or equal to $L$. We can thus bound the number of added edges between components by finding a bound on the number of paths of length $\le L$ that have start and ending point in different components in $G_n$.
This means that within distance $\le L$ of our starting point, there must be an edge going to another component. We can thus bound the number of paths by looking at how many points are $L$-close to edges between components. Our graph sequence $\G$ has bounded degree, let's say $d$. So given one edge, there are at most $\sum\limits_{i=1}^{L} d^{i}$ points close to one of the end points. This means that in the graph $G_n$ there are at most $|E_n^{\epsilon}|\left(\sum\limits_{i=1}^{L} d^{i}\right)^2$ paths that start in one component and end in another.

Hence copying the components $A^n_i$ into $\G'$ the number of edges between them might increase but at most by a factor $\left(\sum\limits_{i=1}^{L} d^{i}\right)^2$. So by taking $\epsilon' = \epsilon \left(\sum\limits_{i=1}^{L} d^{i}\right)^{-2}$ we obtain partitions that have the property $\liminf\limits_{n\rightarrow
\infty} \frac{|E_n^{\epsilon'}|}{|V(G_n')|}\le \epsilon$. This shows that $\G'$ is also hyperfinite.
\end{proof}

We now show hyperfiniteness is a coarse invariant.
\begin{lemma}
If $\G$ and $\Hg$ are coarsely equivalent and $\G $ is hyperfinite, then $\Hg$ is too.
\end{lemma}
\begin{proof}
We follow the construction of $\Hg'$, $f'$ and $\Hg''$ as in Lemma \ref{cost1lemma} and refer the reader to that proof for the details.
Note that we can construct an injective coarse equivalence $f': \G \rightarrow \Hg'$ from a coarse equivalence $f:\G\rightarrow\Hg$ by adding extra vertices to $\Hg$. It is trivial to see that $\Hg$ is hyperfinite if and only if $\Hg'$ is. Now we find a $\Hg''$ which is bi-Lipschitz to $\Hg'$ by injecting $\G$ into $\Hg'$, and linking vertices that are not in the image to image points that are close by. Clearly $\Hg''$ is just $\G$ with extra vertices attached, so it is trivially hyperfinite. Moreover $\Hg''\simeq \Hg'$, so by the previous lemma $\Hg'$ is also hyperfinite. This concludes the proof.
\end{proof}

\section{Cost$-1$ behaves multiplicatively for subgroups}\label{costmulttheo}

Since box spaces originating from filtrations of residually finite groups are in fact graph sequences, one can consider the cost of such a group with respect to a given filtration. We show that cost$-1$ is multiplicative with respect to taking finite index subgroups. Here we consider cost with respect to a fixed filtration. It is not known whether the cost originating from a certain group depends on the chosen filtration. We first show Theorem \ref{supercostmult} in a more restrictive setting, where the subgroup is normal and the Farber sequence is in fact a filtration.
\begin{theorem}\label{multcost} 
Take $\Gamma$ a residually finite group and $\Lambda\triangleleft \Gamma$ with $[\Gamma:\Lambda]< \infty$. Consider a filtration $(N_k)_{k\in\mathbb{N}} $ of $\Gamma$ that is contained in $\Lambda$, then
$$ [\Gamma:\Lambda] (c(\square_{N_k} \Gamma)-1) =  c(\square_{N_k} \Lambda)-1.$$ 
\end{theorem}
\begin{proof}
Since $\Gamma$ is finitely generated there exist elements such that $\Lambda=\langle h_1, \dots,  h_n\rangle$ and $\Gamma=\langle h_1,\dots , h_n, g_1, \dots ,g_r\rangle $. Now since $\Lambda\triangleleft \Gamma$ one sees that $\Gamma/\Lambda = \langle \overline{g}_1, \dots , \overline{g}_r\rangle$. Moreover $\forall k>0$ we can interpret $\Gamma/N_k$ as consisting of $[\Gamma:\Lambda]$ copies of $\Lambda/N_k$ wich are linked together by edges that originate from the generators $\{g_i\}_{i=1\dots r}$. We will call the graph $\Lambda/N_k$ the base copy in $\Gamma/N_k$. Consider two vertices in $\Gamma/N_k$, say $\overline{\omega}$ and $\overline{\omega h}$, where $\omega\in \Gamma$ and $h$ is one of the generators of $\Lambda$. Because $\Gamma/\Lambda$ is a group, there exists a path in $g$-labelled edges of length at most $[\Gamma:\Lambda]$ (say $\Pi_{i=1}^s g_i$) such that $\omega(\Pi_{i=1}^s g_i)$ is in $\Lambda$ (so consequently $\overline{\omega}(\overline{\Pi_{i=1}^s g_i})$ is in $\Lambda/N_k$). Similarly we find such a path $(\Pi_{i=1}^t g_i)$ for $\omega h$. Now clearly since $\omega(\Pi_{i=1}^s g_i)$ and $\omega h(\Pi_{i=1}^t g_i)$ are in $\Lambda$ we can find a path between them in $Cay(\Lambda,\{h_1\dots h_n\})$ of the form $\Pi_{i=1}^R h_i$. Note that considering this path in $\Lambda/N_k$ it can only become shorter (see Figure \ref{trapa}).
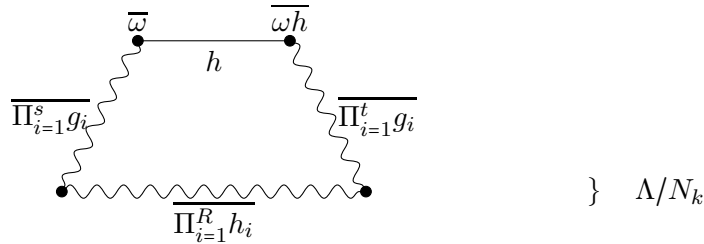
\begin{figure}[H]
	\centering
	\begin{tikzpicture}
	\draw[fill=black] (1,0) circle (2pt) node[align=left,  above] {$\overline{\omega}$}
	(3,0) circle (2pt) node[align=center,   above] {$\overline{\omega h}$}
	(0,-2) circle (2pt)
    (4,-2) circle (2pt);
	\draw   (1,0) -- (3,0) node[midway,  below] {$h$};;
	\draw[style={decorate, decoration=snake}] (1,0)--(0,-2)node[midway,  left] {$\overline{\Pi_{i=1}^s g_i}$};
	\draw[style={decorate, decoration=snake}] (3,0)--(4,-2) node[midway,  right] {$\overline{\Pi_{i=1}^t g_i}$};
    \draw[style={decorate, decoration=snake}] (0,-2)--(4,-2) node[midway,  below] {$\overline{\Pi_{i=1}^R h_i}$};
    \draw (7,-2) node{$\}$};
	\draw (8,-2) node{$\Lambda/N_k$};
	\end{tikzpicture}
	\caption[superthing]
	{Sketch of what happens in $\Gamma/N_k$.}
	\label{trapa}
\end{figure}
Next we take another edge in $\Gamma/N_k$ labelled $h$. Suppose, when repeating the process above, the paths in $g$-generators are the same then we can translate the `trapezoid' in Figure \ref{trapa} and obtain a path of length $\le R$ in the base copy $\Lambda/N_k$. Suppose the paths are not the same, then we continue the process above independently from what we did before and find some new path of some maximal length $R'$. We repeat this process. However note that we only have finitely many ways to go from any point in $\Gamma/N_k$ to the base copy in a $g$-path of length bounded by $[\Gamma:\Lambda]$, and there are only finitely many possibilities for $h$. This means that we can find some superconstant $R_{max}$ by taking the maximum over all found $R$, which bounds the length of all these paths in $\Lambda/N_k$, this for all $k>0$.

This argument shows that removing, in $\square_{N_k} \Gamma$, all edges labelled with an $h$-generator that are not in the base copy $\Lambda/N_k$ gives a graph sequence that is $(R_{max}+2[\Gamma:\Lambda])$-bi-Lipschitz equivalent to $\square_{N_k} \Gamma$.

Now we construct another graph sequence $\G'$ by also deleting $g$-labelled edges. We do this by constructing tree-like structures suspended from the base copy where we make sure that any point not in the base copy is the same distance from the base-copy as before. Clearly any point is $[\Gamma:\Lambda]$-close to $\Lambda/N_k$. 

Now how do we delete $g$-edges? Consider the set of all arcs on $\Lambda/N_k$. These are paths in $\Gamma/N_k$ that have start and end in the base copy, only use $g$-edges and do not pass through the points of the base copy in between. Take then a shortest arc in this set. Suppose this contains an odd number of edges (say $2n+1$). We delete the middle edge (say $(x,y)$). Suppose by doing this another point $z$ becomes further away from the base copy than before. Clearly the shortest path must contain the deleted edge. Hence there exists a path from either $x$ or $y$ to the base copy of length less than or equal to $n-2$. As can be seen in Figure \ref{gedges} this induces an arc of length strictly less than $2n-1$, which is a contradiction to the minimality of the chosen arc. 

A similar argument exists for even length paths. Here we will delete one of the two middle edges. Now we repeat the procedure untill our set of arcs is empty.  Since we showed that this procedure does not change the distance of any vertex to the base copy and every vertex is $[\Gamma:\Lambda]$-close to the base copy, we obtain a graph sequence that is equivalent to the original one. In particular it is of the form $\square_{N_k} \Lambda$ with trees of bounded length suspended from it.

\begin{figure}[H]
	\centering
	\begin{tikzpicture}
	\draw[fill=black] (1,0) circle (2pt) node[left]{$x$}
	(3,0) circle (2pt) node[right]{$y$}
	(0,-2) circle (2pt)
	(4,-2) circle (2pt)
	(2.5,-2) circle (2pt)
	(0,1) circle (2pt) node[left]{$z$};
	\draw   (1,0) -- (3,0) ;
	\draw   (2,0) node[cross=4] {} ;
	\draw[style={decorate, decoration=snake}] (1,0)--(0,-2)node[midway,  left] {$n$};
	\draw[style={decorate, decoration=snake}] (3,0)--(4,-2)node[midway,  right] {$n$};
	\draw[style={decorate, decoration=snake}] (3,0)--(2.5,-2)node[midway,  left] {$<n-1$};
	\draw[style={decorate, decoration=snake}] (0,1)--(1,0)node[midway,  left] {$r$};
	\draw (7,-2) node{$\}$};
	\draw (8,-2) node{$\Lambda/N_k$};
	\end{tikzpicture}
	\caption[superthing]
	{Deleting $g$-edges.}
	\label{gedges}
\end{figure}
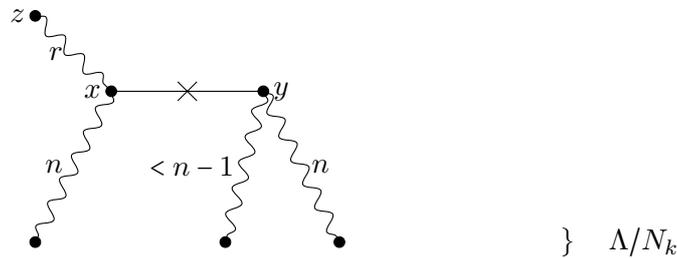

 By repeatedly (i.e. at most $[\Gamma:\Lambda]$ times) using Lemma \ref{lemma1} we see that minimising the edge number  (i.e. finding the cost)  of $\G'$ and thus $\square_{N_k} \Gamma$ is the same as minimising the edge number of $\square_{N_k} \Lambda$. We can take $\forall \epsilon > 0$ a graph sequence $\Hg_\epsilon$ such that $e(\Hg_\epsilon)\le c(\square_{N_k} \Lambda)+\epsilon$ and $\Hg_\epsilon\simeq \square_{N_k} \Lambda $. Constructing $\G_\epsilon$ by introducing $\Hg_\epsilon$ into $\G'$ the infimum of the edge-numbers will give the cost of $\G'$. So let us calculate this:
\begin{equation*}
\begin{aligned}
e(\mathcal{G}_\epsilon) &= \liminf_{n\rightarrow\infty} \frac{|E(G_{\epsilon,n})|}{|V(G_{\epsilon,n})|}\\
&= \liminf_{n\rightarrow\infty} \frac{|E(H_{\epsilon,n})| + ([\Gamma:\Lambda]-1) |V(H_{\epsilon,n})| }{ [\Gamma:\Lambda] |V(H_{\epsilon,n})|}\\
&= \frac{[\Gamma:\Lambda]-1}{[\Gamma:\Lambda]}  + \frac{e(\Hg_\epsilon)}{[\Gamma:\Lambda]}. \\
\end{aligned}
\end{equation*}

This implies that
\[ [\Gamma:\Lambda](e(\G_\epsilon)-1) = e(\Hg_\epsilon)-1 \]
and taking the infimum of both sides for $\epsilon$ we obtain the result
\[ [\Gamma:\Lambda](c(\square_{N_k} \Gamma)-1) = c(\square_{N_k} \Lambda)-1. \]

\end{proof}

\begin{theorem}\label{intermediatemultcost}
Take $\Gamma$ a residually finite group and $\Lambda\le \Gamma$ with $[\Gamma:\Lambda]< \infty$. Consider a filtration $(N_k)_{k\in\mathbb{N}} $ of $\Gamma$ that is contained in $\Lambda$, then
$$ [\Gamma:\Lambda] (c(\square_{N_k} \Gamma)-1) =  c(\square_{N_k} \Lambda)-1.$$ 
\end{theorem}
\begin{proof}
Apply Theorem \ref{multcost} first for $\Gamma$ and $N_1$ and secondly for $\Lambda$ and $N_1$. The result follows.
\end{proof}

\begin{theorem}
	Take $\Gamma$ a residually finite group and $\Lambda\le \Gamma$ with $[\Gamma:\Lambda]< \infty$. Consider a Farber sequence $(N_k)_{k\in\mathbb{N}} $ of $\Gamma$ that is contained in $\Lambda$. Define $\G = \{Sch(\Gamma,N_k,S)\}_{k\in\mathbb{N}}$ and $\Hg= \{Sch(\Lambda,N_k,S')\}_{k\in\mathbb{N}} $ then
	$$ [\Gamma:\Lambda] (c(\G)-1) =  c(\Hg)-1.$$ 
\end{theorem}
\begin{proof}
Since the $N_k$ form a Farber sequence, the induced graph sequence Benjamini-Schramm converges to $Cay(\Gamma)$ and $Cay(\Lambda)$. Hence for large enough $k$, the Schreier graphs behave locally in the same way as a filtration, almost everywhere. Hence we can ignore the small part on which the behaviour deviates and apply the same techniques as in Theorem \ref{multcost} on the remaining large part of the graphs. We need to remark that we are in the more general case where $\Lambda$ is not necessarily normal. However the reader can verify that the proof of Theorem \ref{multcost} does not need that $\Gamma/\Lambda$ is a group, but that it is sufficient that $Sch(\Gamma, \Lambda, S)$ is connected and that $Cay(\Gamma, S)$ is transitive, which is the case. One obtains the same result.
\end{proof}

\section{Property A implies hyperfiniteness}\label{sectionA}

We show that property A implies hyperfiniteness for graph sequences. We adopt the following definition for property A and steal Lemma \ref{stolenwil} and Lemma \ref{stolenwilpr} from Rufus Willett \cite{willett2011property}.

\begin{definition}\label{defA}
Let $\G$ be a graph sequence. Then $\G$ has property A if for all $\epsilon>0$, there exists $S>0$ and a function $\xi: \G\rightarrow l^1(\G)$, denoted $x\mapsto \xi_x$, such that:
\begin{enumerate}
\item $\lVert \xi_x\rVert_1 =1 $ for all $x\in \G$;
\item $ \xi_x$ is supported in the ball of radius $S$ about $x$ for all $x\in \G$; 
\item $\lVert \xi_x-\xi_y \rVert_1 <\epsilon  $ for all $(x,y)\in E(\G)$.
\end{enumerate}
\end{definition}

\begin{lemma}\cite[Lemma 3.2]{willett2011property}\label{stolenwil}
	Let $G$ be a graph with all vertices of degree at most some constant $D$, and let $\epsilon>0$. Say $\phi\in l^1(G)$ is a finitely supported function of norm one such that 
	\[ \sum_{(x,y)\in E(G)} \lvert \phi(x)-\phi(y)\rvert < \epsilon. \]
	Then there exists a subset $F$ of the support of $\phi$ such that $\lvert \partial F\rvert< \epsilon \lvert F\rvert$.
\end{lemma}
\begin{proof}
\end{proof}

\begin{lemma}\label{stolenwilpr}
If a graph sequence $\G$ has property A, then for every $\epsilon>0$ there exist $N,S \in\mathbb{N}$ such that for $n\ge N$, there exist $F_n\subset G_n$ and $x_n\in G_n$ such that $F_n\subset B(x_n,S)$ and $\lvert\partial F_n\rvert < \epsilon \lvert F_n\rvert$. \label{lemA}
\end{lemma}
The following is taken entirely from \cite{willett2011property} where it is part of a larger proof. We reproduce this to improve readability.
\begin{proof}
We fix $\epsilon>0$ and obtain for $\epsilon/d$ (where $d$ is the degree bounding $\G$) a function $\xi$ and constant $S$ as in \ref{defA}. By possibly taking absolute values, we may assume $\xi_x(y)\ge 0$ for all $x,y\in \G$. For large $n$ (i.e. $n$ greater than or equal to some $N$) one can restrict $\xi$ to $G_n$ and obtain $\xi^n: G_n\rightarrow l^1(G_n)$. We have the following inequality:
\begin{equation}\label{supereq}
\sum_{(x,y)\in E(G_n)} \lVert \xi_x-\xi_y \rVert_1 < d \lvert G_n\rvert \epsilon / d.
\end{equation}
On the other hand we have 
\begin{equation*}
\lvert G_n\rvert = \sum_{x\in G_n} \lVert \xi_x^{n}\rVert_1 = \sum_{x\in G_n}\sum_{z\in G_n} \lvert \xi_x^n (z)\rvert
\end{equation*}
and introducing this into $(\ref{supereq})$ gives
\begin{equation*}
\sum_{z\in G_n}\sum_{(x,y)\in E(G_n)} \lvert \xi^n_x(z)-\xi^n_y(z) \rvert < \sum_{z\in G_n}\epsilon \sum_{x\in G_n} \lvert \xi_x^n (z)\rvert.
\end{equation*}
So there must exist some $z_0\in G_n$ such that 
\begin{equation*}
\sum_{(x,y)\in E(G_n)} \lvert \xi^n_x(z_0)-\xi^n_y(z_0) \rvert < \epsilon \sum_{x\in G_n} \lvert \xi_x^n (z_0)\rvert.
\end{equation*}
Hence if we define $\Psi^n(x) = \xi^n_x (z_0) $, we see that $\Psi^n/\lVert\Psi^n\rVert_1$ has the necessary properties to apply Lemma \ref{stolenwil}.
\end{proof}

Repeating the previous lemma several times we obtain the following theorem.

\begin{theorem}\label{Aimplieshyperfinite}
A graph sequence $\G$ that has property A is also hyperfinite. 
\end{theorem}
\begin{proof}
Fix $\epsilon >0$. Since $\G$ has property A we find $\xi$ and $S$ from Definition \ref{defA}. Take $\G$ and apply Lemma  \ref{lemA}. For large $n$ we obtain sets $F_n\subset B(x_n,S)$ as in the lemma. Constructing $\G'$ by removing the vertices from the sets $F_n$ in $\G$ gives a graph sequence that still  has property A, as we show now. We construct $\xi':\G'\rightarrow l^1(\G')$ on each $G_n\backslash F_n$ by doing the following procedure $\lvert F_n \rvert$ times (one time for every $a\in F_n$).

Take the a first element $a\in F_n$. Consider the set $R_{a} = \{ x\in G_n\backslash F_n \; \lvert\; \xi_x(a)\ne 0  \}$. Note that $R_{a}$ can be decomposed into its connected components (i.e. $R_{a} = \sqcup_{i=1\dots k}A_i^{a}  $  ). Next one takes elements $\mu_1^{a}\dots \mu_k^{a}\in R_{a}$, one for each connected component $A_i^{a}$.
We define the map $\xi^1:G_n\backslash\{a\}\rightarrow l^1(G_n\backslash \{a\})$ by for each $x\in A_i^a$ setting $\xi_{x}^1 (\mu_i^a) = \lvert\xi_{x}(\mu_i^a)\rvert+\lvert\xi_{x}(a)\rvert$ and setting $\xi_{x}^1(y) = \xi_{x}(y)$ if $y\ne \mu_i^a$ for all $i$. For the other elements of $G_n\backslash\{a\}$ the map $\xi^1$ has the same image as $\xi$. Now we repeat the procedure for a second $\overline{a}\in F_n$, but now applied to $\xi^1$. If one continues this for all elements in $F_n$ we obtain some new $\xi'$ such that  $\lVert\xi'_{x}\rVert =1$ for all $x\in G_n\backslash F_n$. Define for $x$ the set $A_x = \{\mu_i^a\;\lvert \; x\in A_i^a \}$ .

 Now 
\begin{align}
\lVert \xi_x'\rVert_1 &= \sum_{z\in G_n\backslash F_n} \lvert \xi'_x(z)\rvert\nonumber\\
&= \sum_{z\in G_n\backslash (F_n\cup A_x)} \lvert \xi_x'(z) \rvert + \sum_{z\in A_x} \lvert \xi_x'(z)\rvert\nonumber\\
&= \sum_{z\in G_n\backslash (F_n\cup A_x)} \lvert \xi_x(z) \rvert + \sum_{\mu_i^a \in A_x} \lvert \xi_x(\mu_i^a) + \xi_x(a)\rvert\nonumber\\
&= \sum_{z\in G_n\backslash F_n} \lvert \xi_x(z) \rvert + \sum_{a\in F_n} \lvert \xi_x(a)\rvert\label{eqwilletta}\\
&= \lVert \xi_x\rVert_1,\nonumber
\end{align}
where equality (\ref{eqwilletta}) holds since any $x$ is at most in one connected component $A_i^a$ for each $a$.

Moreover if $(x,y)\in E(G_n\backslash F_n) $, then 
\begin{align}
\lVert \xi_x'-\xi_y'\rVert_1 &= \sum_{z\in G_n\backslash F_n} \lvert \xi'_x(z) -\xi'_y(z)\rvert \nonumber\\ 
&\le  \sum_{z\in G_n\backslash (F_n\cup A_x\cup A_y)} \lvert \xi_x(z) -\xi_y(z)\rvert + \sum_{z\in A_x\cup A_y} \lvert \xi_x'(z) -\xi_y'(z)\rvert \nonumber\\
&\le  \sum_{z\in G_n\backslash (F_n\cup A_x\cup A_y)} \lvert \xi_x(z) -\xi_y(z)\rvert\nonumber\\
& \; \;\;\;\;\; +  \sum_{z\in A_x\cup A_y} \lvert \xi_x(z) -\xi_y(z)  +\sum_{\mu_i^a\in A_x \cup A_y \lvert \mu_i^a = z} \xi_x(a) -\xi_y(a)  \rvert\label{eqwillett1} \\
&\le  \sum_{z\in G_n\backslash F_n} \lvert \xi_x(z) -\xi_y(z)\rvert\nonumber\\
&= \lVert \xi_x-\xi_y\rVert_1.\nonumber
\end{align}
One can see inequality (\ref{eqwillett1}) holds by the following argument: if $\xi_x(a)$ and $\xi_y(a)$ are both non-zero, then $x$ and $y$ are in the same connected component and thus our construction of $\xi'$ moves the values $\xi_x(a)$ and $\xi_y(a)$ from $a$ to the same point $\mu_i^a$. 

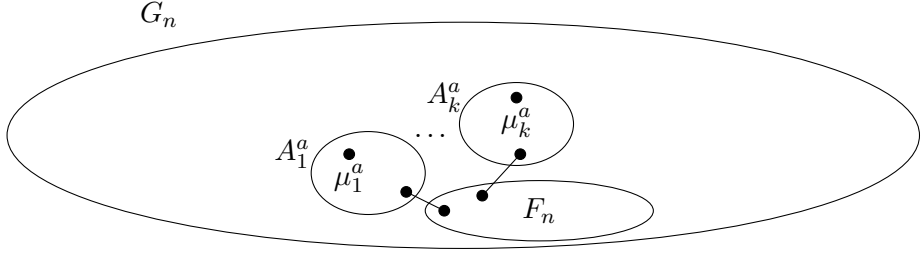
\begin{figure}[H]
	\centering
	\begin{tikzpicture}
	\draw (0,0) ellipse (6cm and 1.5cm);
	\draw (1,-1) ellipse (1.5cm and 0.4cm);
	\draw (1,-1) node{$F_n$};
	\draw (-4,1.6) node{$G_n$};
	\draw[fill=black] (-0.25,-1) circle (2pt); 
	\draw   (-0.25,-1) -- (-0.75,-0.75);
	\draw[fill=black] (-0.75,-0.75) circle (2pt); 
	\draw (-1.25,-0.5) ellipse (0.75cm and 0.55cm);
	\draw[fill=black] (-1.5,-0.25) circle (2pt) node[align=center,   below]{$\mu_1^a$};
	\draw (-2.25,-0.25) node{$A_1^a$}; 
	\draw (0.7,0.15) ellipse (0.75cm and 0.55cm);
	\draw[fill=black] (0.25,-0.80) circle (2pt); 
	\draw   (0.25,-0.80) -- (0.75,-0.25);
	\draw[fill=black] (0.75,-0.25) circle (2pt); 
	\draw[fill=black] (0.7,0.5) circle (2pt) node[align=center,   below]{$\mu_k^a$};
	\draw  (-0.25,0.5) node{$A_k^a$};
	\draw  (-0.4,0) node{$\dots$};
	\end{tikzpicture}
	\caption[superthing]
	{Construction of  $R_a$.}
\end{figure}	

Now let us consider the support of an arbitrary $\xi_x'$. Of course it is possible that by removing the points of $F_n$ we either disconnect the graph or make points in $G_n\backslash F_n$ be much farther apart than was originally the case. For some $a\in G_n\backslash F_n$ define once again $R_{a} = \{ x\in G_n\backslash F_n \; \lvert\; \xi_x(a)\ne 0  \}$. We once again find the connected components and choose elements $\mu_i$ in these components. We define $\xi''$ in a similar way as $\xi'$ and can show that $\xi''$ still has the necessary properties. Moreover any $\xi_x''$ has support within $B(x,\max\limits_{y\in \G} \lvert B(y,S) \rvert)$. This is guaranteed by its construction. However this value is bounded, since the degree of the graph sequence is uniformly bounded.

So if $supp(\xi_x)\subset B(x,S)$, then $supp(\xi'_x)\subset B(x,\max\limits_{x\in \G} \lvert B(x,S) \rvert)$. Hence $\xi'$ has all properties needed in definition \ref{defA}. If we can once again apply Lemma \ref{lemA}, we obtain new $F'_n\subset B(x_n',\max\limits_{x\in \G} \lvert B(x,S) \rvert)\subset G_n'$. Next, one looks at $\G'' = \{G_n\backslash(F_n\cup F_n')\}_{n\in \mathbb{N}}$. Further repeating the procedure, we obtain, for $n\ge N$, a partition of the graphs $G_n$ with partitioning sets inside some ball of radius $\max\limits_{x\in \G} \lvert B(x,S) \rvert$. Note that the value $N$ stays unchanged every time the procedure is repeated, since this value only depends on $\xi$. So by construction we obtain partitions (the sets $F_n$, $F_n'$ etc.) of the graphs $G_n$ with $n\ge N$ which defines hyperfiniteness.
\end{proof}

\section{Elek's theorem for Farber sequences}\label{elekfarbersection}

Elek shows in \cite{elekcombcost} that amenability of a residually finite group is equivalent to hyperfiniteness of box spaces originating from filtrations of the group and equivalent to those box spaces having property A. It is believed that what holds for filtrations also holds for Farber sequences and their associated graph sequences of Schreier graphs. This motivates us to prove Elek's theorem but for Farber sequences.

A filtration approximates a given group by taking a sequence of normal subgroups with trivial intersection, more specifically it provides a Benjamini-Schramm convergent sequence of Cayley-graphs. For a Farber sequence we allow any finite index subgroup and ask that the associated graph sequence of Schreier graphs is Benjamini-Schramm convergent. We use this characterisation later on.

\begin{lemma}\label{lemmaamimpA}
If a finitely generated group $\Gamma$ is amenable, then any graph sequence consisting of Schreier graphs of this group has property A.
\end{lemma}
\begin{proof}
Fix $\epsilon> 0$. Take $\Lambda\le \Gamma$ and a generating set for $\Gamma$, this defines a Schreier graph on the right cosets of $\Lambda$. Since $\Gamma$ is amenable there exists a F\o lner set $F\subset \Gamma$ such that $|\partial F|<\epsilon|F| $, where $\partial F$ denotes the boundary of $F$. Now we define the following map:
\[ \xi: \Gamma\rightarrow \ell_1(\Gamma): x \mapsto \left( g\mapsto  \left\{
    \begin{array}{ll}
        \frac{1}{|F|} & \mbox{if } g\in xF \\
        0 & \mbox{otherwise}
    \end{array}
\right. \right)   ,\]
i.e. $\xi_x$ is the normalised characteristic function of $xF$. Clearly $\lVert\xi_x\rVert_1 = 1$ and of finite support by the finiteness of $F$. Furthermore $\lVert\xi_x-\xi_y\rVert<\epsilon$ if $(x,y)$ is an edge in $Cay(\Gamma)$.

Now we will show that we can find $\overline{\xi}: \Gamma/\Lambda\rightarrow \ell_1(\Gamma/\Lambda)$ with support smaller or equal than the support of $\xi$. More precisely it has the properties as in \ref{defA} for our previously chosen $\epsilon$.
We define $\overline{\xi}$ by
\[ \overline{x}\rightarrow \left( \overline{y}\mapsto  \sum_{z\in \overline{y}\cap xF  }  \frac{1}{|F|}   \right)  . \]
Here $z$ is obviously in $\Gamma$. We check this map is well-defined (i.e. that the image of $\overline{x}$ does not depend on the chosen representative). Suppose $\overline{x} = \overline{x'}$. Hence there exists an element $h\in \Lambda$ such that $hx = x'$. One can calculate the following: $\lvert \overline{y}\cap x'F \rvert = \lvert h\overline{y}\cap hxF \rvert = \lvert h(\overline{y}\cap xF)\rvert = \lvert\overline{y}\cap xF\rvert$. So clearly $\overline{\xi}_{\overline{x}} = \overline{\xi}_{\overline{x'}}$. Also observe that $\lVert \overline{\xi}_{\overline{x}}\rVert=1$.

Let us take a look at $\lVert \overline{\xi}_{\overline{x}} - \overline{\xi}_{\overline{y}}\rVert$, where $\overline{x} = \overline{y}g$ with $g$ some generator of $\Gamma$. Note, that here we are looking at points that are linked by an edge in the Schreier graph of $\Gamma/\Lambda$.
There exist representatives $x$ of $\overline{x}$  and $y$ of $\overline{y}$ such that $x=yg$ in $\Gamma$. Hence
\begin{align}
 \lVert \overline{\xi}_{\overline{x}} - \overline{\xi}_{\overline{y}}\rVert
&= \frac{1}{|F|} \sum_{\Lambda z\in \Gamma/\Lambda} \lvert \, \lvert \Lambda z\cap xF\rvert - \lvert \Lambda z\cap yF\rvert \, \rvert\nonumber \\
&\le \frac{1}{|F|} \sum_{\Lambda z\in \Gamma/\Lambda} |(xF\Delta yF)\cap \Lambda z|\nonumber \\
&= \frac{| xF\Delta yF|}{|F|}\label{eqsymm} \\
&\le \epsilon.\label{ineqeps}
\end{align}

Here equality (\ref{eqsymm}) follows from the $\Lambda z$ forming a partition of $\Gamma$ and inequality (\ref{ineqeps}) follows from the choice of $x$ and $y$ (i.e. being connected by an edge in the Cayley graph of $\Gamma$).

This proves the lemma.
\end{proof}

We show that hyperfiniteness of a sofic approximation implies amenability of the group.
\begin{theorem}\label{sofapphypam}
Take a finitely generated group $\Gamma$ and a sofic approximation $\mathcal{G} = \{G_n\}_{n\in\mathbb{N}}$. If $\G$ is hyperfinite, then $\Gamma$ is amenable.
\end{theorem}
\begin{proof}
We show that hyperfiniteness of a sofic approximation implies amenability for $\Gamma$. Now since $\mathcal{G}$ is hyperfinite there exists an $R>0$ and for every $n$ a partition $\{A_i^n\}_{i=1\dots k_n}$, where $|A_i^n|<R$. Moreover if $E_n^\epsilon$ is the collection of edges between partitioning sets in $G_n$, then we have that $\liminf_{n\rightarrow \infty} \frac{E_n^{\epsilon}}{V_n} < \epsilon$. From here on we take such a partition with respect to $\epsilon/4$.
Now since $\mathcal{G}$ is a sofic approximation (and thus the graph sequence is equivalently Benjamini-Schramm convergent) the number of vertices in $G_n$ for which the $2R$-ball centered at this vertex is different from the $2R$-ball in the Cayley-graph of $\Gamma$ will proportionally go to zero. Take $V_n^{2R}$ the vertices in $G_n$ that have a non-isomorphic $2R$-ball. So we can now choose $N>0$  such that $\frac{|V_N^{2R}|}{|V_N|} < 1/2$. 

We say a vertex is nice if the ball $B(x,R)$ is isomorphic to the ball $B(e,R)$ in the Cayley graph of $\Gamma$. We say a partitioning set $A_i^N$ is particularly nice if all its vertices are nice. This is interesting because then the set can be lifted to the Cayley graph of $\Gamma$. Note that if $A_i^N$ contains one point of $V_N\backslash V_N^{2R}$, then all its points are nice. Hence the union of all particularly nice partitioning sets contains at least $\lvert V_N\backslash V_N^{2R}\rvert$ vertices. We can calculate the following.
\begin{equation*}
\begin{aligned}
\frac{|E_N^{\epsilon/4}|}{|V_N\backslash V_N^{2R}|}
\le \frac{|E_N^{\epsilon/4}|}{|V_N| (1-1/2)}
< \frac{\epsilon}{4} \cdot 2
= \frac{\epsilon}{2}.
\end{aligned}
\end{equation*}
Hence there exists some particularly nice $A_I^N$ that has boundary smaller than $\epsilon$ times its number of points. Moreover this set is situated in the part of the graph for which the $R$-balls are isomorphic to the $R$-balls of the Cayley graph of $\Gamma$ and thus this set can be lifted to $\Gamma$, defining a F\o lner set with respect to $\epsilon$.
If we do this for all $\epsilon$, we obtain a F\o lner sequence for $\Gamma$. This concludes the proof.
\end{proof}

Now the hard work is done and our main theorem becomes an immediate corollary of our previous theorems.

\begin{theorem}\label{elekgeneral}
	Consider a finitely generated group $\Gamma$ and a Farber sequence $(N_k)_{k\in \mathbb{N}}$. If $\mathcal{G}$ is the associated graph sequence of Schreier graphs then the following are equivalent:
	\begin{enumerate}
		\item $\Gamma$ is amenable
		\item $\mathcal{G}$ has property A
		\item $\mathcal{G}$ is hyperfinite.
	\end{enumerate}
\end{theorem}
\begin{proof} 
We prove $1 \Rightarrow 2$, $2\Rightarrow 3$, $3\Rightarrow 1$. The three implications are particular cases of respectively \ref{lemmaamimpA}, \ref{Aimplieshyperfinite} and \ref{sofapphypam}.
\end{proof}

Our results are optimal in the following sense. Alekseev and Finn-Sell show in \cite{alekseev2016sofic} that if a sofic approximation has property A, then the group must be amenable. They ask whether the converse holds. The answer is negative. For an easy counterexample we introduce the notion of expanders. 

\begin{definition}
For a finite graph $G$, the Cheeger constant is defined as 
\[h(G) := \min \left\{ \frac{\lvert \partial A\rvert}{\lvert A \rvert} \; \lvert \; A\subset V(G),\; 0< \lvert A\rvert \le \frac{1}{2}\lvert V(G) \rvert  \right\} .\]
A graph sequence $\G=\{G_n\}_{n\in\mathbb{N}}$ is an expander if there is some $c>0$ such that for every $n\in\mathbb{N}$ the Cheeger constant $h(G_n)> c$.
\end{definition}

It is a well-known fact that expanders do not have property A (see for example \cite{MR2562146}). To answer our question, consider the following scenario. Take an amenable residually finite group $\Gamma$ and a box space $\square_{N_k} \Gamma$ of this group with respect to a generating set $S$. This Benjamini-Schramm converges to the Cayley graph of $\Gamma$. Now take an expander sequence $\{H_n\}_{n\in\mathbb{N}}$ and fix a labeling in $S$ on each $H_n$. For each $H_n$, there exists a graph $Cay(\Gamma/N_k,\overline{S})$ in the box space of $\square_{N_k} \Gamma$ such that $\frac{\lvert V(H_n)\rvert}{\lvert V(\Gamma/N_k)\rvert}< \frac{1}{n}$. Now we choose one point of $H_n$ and one point of $Cay(\Gamma/N_k,\overline{S})$ and identify them, obtaining the graph $G_n'$. Now we construct the following graph sequence $\G = \{ G_n'\}$. One sees that the expander sequence that we added to the box space is negligible with respect to the box space. As such this still Benjamini-Schramm converges and still forms a sofic approximation for $\Gamma$. However this no longer has property A since it contains an expander. Hence there is no hope to show Theorem \ref{elekgeneral} for arbitrary sofic approximations. 

Another question asked is whether hyperfiniteness implies property A for sofic approximations. The previous example also gives a negative answer in this case. So also here we have found the limits of the implication. In Theorem \ref{sofapphypam} we showed that hyperfiniteness of the sofic approximation implies amenability of the group. The other implication also holds. This is due to Elek in \cite{elekcombcost}. The original theorem only mentions that this is true for box spaces, however the exact same proof can be used to obtain the same result for sofic approximations.\newline

We now introduce a new concept property almost-A in order to `fix' the theorem. 

\begin{definition}\label{almosta}
Let $\G$ be a graph sequence, then $\G$ has property almost-A if for all $n\in\mathbb{N}$, there exists sets $V_n' \subset V(G_n)$ such that:
\begin{enumerate}
	\item $\lim\limits_{n\rightarrow\infty} \frac{\lvert V_n'\rvert}{\lvert V(G_n)\rvert} = 0$;
	\item $ \mathcal{G}' := \{ G_n\backslash V_n' \}_{n\in\mathbb{N}}$ has property A.
\end{enumerate}
\end{definition}
The following theorem saves the equivalences for sofic approximations and in doing so, it gives an equivalent definition for amenability.
\begin{theorem}
Consider a group $\Gamma$ and a sofic approximation $\G$ of that group, then the following are equivalent:
\begin{enumerate}
	\item $\Gamma$ is amenable;
	\item $\mathcal{G}$ has property almost-A;
	\item $\G$ is hyperfinite.
\end{enumerate}
\end{theorem}
\begin{proof}
Firstly if $\mathcal{G}$ has property almost-A, then the graph sequence $\G'$ in Definition \ref{almosta} is still a sofic approximation of the group $\Gamma$. Hence by the theorem of Alekseev and Finn-Sell (see \cite{alekseev2016sofic}), $\Gamma$ is amenable.\newline
On the other hand for a sofic approximation $\G$ of an amenable group $\Gamma$ we define $\phi: \mathbb{N}\rightarrow\mathbb{N}$. We take for $\phi(s)$ the smallest natural number such that for all $m\ge \phi(s)$ the probability 
$$p_s(G_m,B_s(Cay(\Gamma),e)) \ge 1-\frac{1}{s}.$$
This function is well-defined by Benjamini-Schramm convergence of $\G$ and is increasing. We define the following for $\phi(s)\le n< \phi(s+1)$
$$V_n' = \{ x\in V(G_n) \; \lvert\;  B_{s}(G_n,x) \nsimeq B_s(Cay(\Gamma),e) \} .$$ 
We show $\G$ has property almost-A with respect to the sets $V_n'$ and thus the induced graph sequence $\mathcal{G}' := \{ G_n\backslash V_n' \}_{n\in\mathbb{N}}$ must have property A. Fix $\epsilon>0$. Since $\Gamma$ is amenable we have a F\o lner set $F$ such that $\lvert \partial F\rvert < \epsilon \lvert F\rvert$. Take $S = \lvert F\rvert$. We define $\xi: \G' \rightarrow  \ell^1(\mathcal{G})$ for $x \in G_n'$ where $n\ge \phi(S)$ by 

$$ x \mapsto \left( y\mapsto  \left\{
\begin{array}{ll}
\frac{1}{S} & \mbox{if } y\in xF \\
0 & \mbox{otherwise}
\end{array}
\right. \right) ,  $$
and otherwise for $x\in G_n'$ and $n< \phi(S)$
$$ x \mapsto \left( y\mapsto  \left\{
\begin{array}{ll}
\frac{1}{\lvert V(G_n')\rvert} & \mbox{if } y\in G_n' \\
0 & \mbox{otherwise}
\end{array}
\right. \right) .   $$
Note that the expression $xF$ is well-defined by the isomorphism $B_{S}(G_n,x) \simeq B_S(\Gamma,e)$. We can extract from $\xi$ some $\xi': \G' \rightarrow \ell^1(\mathcal{G}')$ that has the necessary properties for property A by applying the method we used in the proof of Theorem \ref{Aimplieshyperfinite}.
\end{proof}

\section*{Acknowledgements}
We wish to thank L\'aszl\'o M\'arton T\'oth for pointing out that Theorem \ref{intermediatemultcost} could be generalised to Farber sequences. We are also immensely grateful to Ana Khukhro, and Alain Valette for their comments on an earlier version of the manuscript.

\bibliographystyle{plain}
\bibliography{cost}

\end{document}